\documentclass{amsart}
\usepackage {amsmath, amscd}
\usepackage {amssymb}
\usepackage {epsfig}
\usepackage {bm}
\usepackage {indentfirst} 
\usepackage{color}
\usepackage{tikz-cd}

\usepackage[active]{srcltx}

\numberwithin{equation}{section}

\def\<{\langle}
\def\>{\rangle}

\def\LL{{\mathcal L}}

\def\bbD{\mathbb{D}}

\def\1{\mathbf{1}}

\newtheorem{lem}{Lemma}[section]
\newtheorem{prop}[lem]{Proposition}
\newtheorem{thm}[lem]{Theorem}
\newtheorem{corollary}[lem]{Corollary}

\theoremstyle{definition}
\newtheorem{rmk}[lem]{Remark}

\title{The Generalized Crofoot Transform }

\author{Rewayat Khan}
\address{Abdus Salam School of Mathematical Sciences, GC University, Lahore, Pakistan}
\email{rewayat.khan@gmail.com}


\begin{document}
	\maketitle

\begin{abstract}
	 We introduce a generalized Crofoot transform between the model spaces corresponding to matrix-valued  inner functions.
 As an application, we obtain results about matrix-valued truncated Toeplitz operators.
\end{abstract}

\section{Introduction}

The theory of completely nonunitary contractions on a Hilbert space, as developed in~\cite{NF}, 
provides functional models for arbitrary completely nonunitary contractions. 
In the particular case when the dimensions of the defect spaces of the contraction (to be defined below) is 1 and the contraction is stable, the model space is the function space $ H^2\ominus \theta H^2 $, where $ H^2 $ is the Hardy--Hilbert space and $ \theta $ is an inner function. These spaces are often called shortly \emph{model spaces} and have been the object of extensive study in the last decades. In particular, a direction of study initiated in~\cite{Sa} deals with the so-called \emph{truncated Toeplitz operators}, which are compression to model spaces of multiplication operators. 
The Crofoot transform, introduced in~\cite{Cr}, is a useful tool for transferring properties between model spaces and between the associated spaces of truncated Toeplitz operators.

A more general type of model space is obtained when the scalar inner function is replaced by a matrix-valued inner function $ \Theta $. Then the space $K_\Theta=H^2(E)\ominus \Theta H^2(E)$, with $E$ a
finite dimensional Hilbert space. In this context, matrix valued truncated Toeplitz operators and their properties
has been formally introduced in \cite{RT}. 

The current paper introduces the generalization of the Crofoot transform to the model spaces associated to matrix-valued inner functions. As an application, we investigate the behaviour of the space of matrix-valued truncated Toeplitz operators with respect to this transformation.

The structure of the paper is the following.
After a section of general preliminaries about spaces of vector and matrix valued functions, we give a primer of the
properties of the vector-valued  model spaces and models operators. The generalized Crofoot transformation and its link to matrix valued truncated Toeplitz operators is defined
in Section 3. In Section 4 we investigate the case when the matrix-valued inner function is complex symmetric.
 
One should note that the generalized Crofoot transform that we introduce is related to the study of perturbations  of contractions as appearing in~\cite{BL, BT1, BT2, Fu}. However, we work here in a concrete framework and we obtain explicit results for all the transformations involved.

\section{Preliminaries}

Let $\mathbb{C}$ denote the complex plane, $\mathbb{D}={\{z\in\mathbb{C}: |z|<1}\}$ the unit disc, $\mathbb{T}={\{z\in\mathbb{C}:|z|=1}\}$ the unit circle.
Throughout the paper $\mathbb{C}^{d}$ will denote $d$ dimensional complex Hilbert space, and $\LL(\mathbb{C}^d)$ the algebra of bounded linear operators on $\mathbb{C}^{d}$, which  may be identified with $d\times d$ matrices.

The space $L^{2}(\mathbb{C}^{d})$  is defined,  as usual, by
\[
L^{2}(\mathbb{C^{d}})={\Big\{f:\mathbb{T}\to \mathbb{C}^{d}: f(e^{it})=\sum\limits_{-\infty}^{\infty}a_{n}e^{int}:  a_{n}\in \mathbb{C}^{d},\quad   \sum\limits_{-\infty}^{\infty}\|a_{n}\|^{2}<\infty }\Big\},
\]
endowed with the inner product
\begin{equation*}
\langle f,g\rangle_{L^{2}(\mathbb{C}^{d})}=\frac{1}{2\pi}\int\limits_{0}^{2\pi}\langle f(e^{it}),g(e^{it})\rangle_{\mathbb{C}^{d}}\,dt.
\end{equation*}

The Hardy space  $H^{2}(\mathbb{C}^{d})$ is the subspace of $L^2(\mathbb{C}^{d})$ formed by the functions with vanishing negative Fourier coefficients; it can be identified with a space of $\mathbb{C}^{d}$-valued  functions analytic in $\bbD$, from which  the boundary values can be recovered almost everywhere through radial limits.

Let $S$ denote the forward shift operator $(Sf)(z)=zf(z)$ on $H^{2}(\mathbb{C}^{d})$; it is the restriction of $ M_z  $, the multiplication with the variable $ z $, to $H^{2}(\mathbb{C}^{d})$. Its adjoint (the backward shift) is the operator \[(S^{*}f)(z)=\frac{f(z)-f(0)}{z}.
\]

An \emph{inner function} will be an element $\Theta\in H^2(\LL(\mathbb{C}^{d})$ whose boundary values are almost everywhere unitary operators (equivalently, isometries or coisometries) in $\LL(\mathbb{C}^d)$.
 All niner functions in the  sequel are assumed to be pure, that is $ \|\Theta(0)\|<1 $.

The  \emph{model space} associated to $\Theta$, denoted by $K_{\Theta}$,  is defined by $K_{\Theta}=H^{2}(\mathbb{C}^{d})\ominus \Theta H^{2}(\mathbb{C}^{d})$; the orthogonal projection onto $K_\Theta$ will be denoted by $P_\Theta$. The properties of the model space are familiar to many analysts in the scalar case.  On the other hand, the vector valued version is less widely known (despite playing an important role in the Sz.-Nagy--Foias theory of contractions~\cite{NF}).

The
 model space $K_{\Theta}$  is  a vector valued reproducing kernel Hilbert space; its reproducing  kernel function, which takes values in $\LL(\mathbb{C}^d)$, is
 \[
 k_{\lambda}^{\Theta}(z)=
 \frac{1}{1-\overline{\lambda}z}(I-\Theta(z)\Theta(\lambda)^{*}).
 \]
 This means that for any $x\in \mathbb{C}^{d}$ we have $k_{\lambda}^{\Theta}x\in K_{\Theta}$, and, if $f\in K_\Theta$, then
 \[
 \langle f, k_{\lambda}^{\Theta}x\rangle_{K_\Theta}=\< f(\lambda),x \>_{\mathbb{C}^{d}}.
 \]
We will  also have the occasion to consider a related family of functions, namely
\[
 \widetilde{k_{\lambda}^{\Theta}}(z)=\frac{1}{z-\lambda}(\Theta(z)-\Theta(\lambda
 )).
\]

The \emph{model operator} $S_{\Theta}\in\LL(K_{\Theta})$ is defined  by the formula
\begin{equation}\label{eq:Model Operator}
(S_{\Theta}f)(z)=P_{\Theta} (zf), \quad f\in K_{\Theta}.
 \end{equation}

The adjoint of $S_{\Theta}$  is given by
\[
(S_{\Theta}^{*}f)(z)=\frac{f(z)-f(0)}{z};
\]
it is the restriction of the left shift in $H^{2}(\mathbb{C}^d)$ to the $S^{*}$-invariant subspace $K_{\Theta}$.
 The action of $S_{\Theta}$ is more precisely described if we introduce the following subspaces of $K_{\Theta}$ (the defect spaces of $S_{\Theta}$ in the terminology of~\cite{NF}):
\begin{equation}\label{eq:definition of defect spaces}
\begin{split}
\mathcal{D}_{*}&=\Big\{\frac{1}{z}(\Theta(z)-\Theta(0))x:x\in \mathbb{C}^d \Big\}\\
\mathcal{D}&=\{(I-\Theta(z)\Theta(0)^{*})x:x\in \mathbb{C}^{d}\}.
\end{split}
\end{equation}
The action of $S_{\Theta}$ on $\mathcal{D}^{\perp}$, $\mathcal{D}$ and of $S_{\Theta}^{*}$ on $\mathcal{D}_{*}^{\perp}$, $\mathcal{D}_{*}$, are given by the formula's below:
\begin{equation}\label{eq: S Theta}
\begin{split}
 (S_{\Theta}^{*}f)(z)=\left\{
                     \begin{array}{ll}
                       \frac{f(z)}{z} ~~~ ~~~~~~~~~~~~~~~~~~~~~~~~~~~~~~~~ for ~f\in D^{\perp}, \\
                       -\frac{1}{z}\big(\Theta(z)-\Theta(0)\big)\Theta(0)^{*}x ~~~~~~~~~ for ~f=\big(I-\Theta(z)\Theta(0)^{*}\big)x\in D;
                     \end{array}
                   \right.\\
 (S_{\Theta}f)(z)=\left\{
                     \begin{array}{ll}
                       zf(z) ~~~ ~~~~~~~~~~~~~~~~~~~~~~~~~~~~~for ~f\in D_{*}^{\perp}, \\
                       -\big(I-\Theta(z)\Theta(0)^{*}\big)\Theta(0)x ~~~~~~~~~ for ~f=\frac{1}{z}\big(\Theta(z)-\Theta(0)\big)x\in D_{*}.
                     \end{array}
                   \right.
                   \end{split}
\end{equation}

We will use the following standard notation.  If $T\in \mathcal{L}(E)$
is a contraction, then the operators $D_{T}=(I-T^{*}T)^{\frac{1}{2}}$ and $D_{T^{*}}=(I-TT^{*})^{\frac{1}{2}}$ are called the defect operators
and $\mathcal{D}_{T}=\overline{D_{T}E}$ and $\mathcal{D}_{T^{*}}=\overline{D_{T^{*}}E}$ are called the defect spaces of $T$.

\section {Generalized Crofoot Transform }\label{se:GCT}
Let $\Theta(\lambda):\mathbb{C}^d\longrightarrow \mathbb{C}^d$ be a pure inner function and  $W$ a fixed strict contraction  acting on $\mathbb{C}^d$.
\begin{prop}
The function $\Theta^{\prime}$ defined in terms of inner function $\Theta$ and strict contraction $W$ given by
\begin{equation}\label{eq:formula for Theta'}
\Theta^{'}(\lambda)=-W+ D_{W^{*}}(I-\Theta(\lambda) W^{*})^{-1}\Theta(\lambda) D_{W}
\end{equation}
is a pure inner function.
\end{prop}
\begin{proof}
Consider
\begin{align*}
\Theta^{\prime}(e^{it})\Theta^{\prime *}(e^{it})
&=[-W+ D_{W^{*}}(I-\Theta(e^{it}) W^{*})^{-1}\Theta(e^{it}) D_{W}]\\
&\qquad[-W^{*}+ D_{W}\Theta(e^{it})^{*}(I- W\Theta(e^{it})^{*})^{-1} D_{W^{*}}]\\
&=WW^{*}-WD_{W}\Theta^{*}(I-W\Theta^{*})^{-1}D_{W^{*}}-D_{W^{*}}(I-\Theta W^{*})^{-1}\Theta D_{W}W^{*}\\
&\qquad+D_{W^{*}}(I-\Theta W^{*})^{-1}\Theta D_{W}^{2}
\Theta^{*}(I-W\Theta^{*})^{-1}D_{W^{*}}\\
&=WW^{*}-D_{W^{*}}W\Theta^{*}(I-W\Theta^{*})^{-1}D_{W^{*}}-D_{W^{*}}(I-\Theta W^{*})^{-1}\Theta W^{*}D_{W^{*}}\\
&\qquad+D_{W^{*}}(I-\Theta W^{*})^{-1}\Theta D_{W}^{2}
\Theta^{*}(I-W\Theta^{*})^{-1}D_{W^{*}}\\
&=WW^{*}-D_{W^{*}}[W\Theta^{*}(I-W\Theta^{*})^{-1}-(I-\Theta W^{*})^{-1}\Theta W^{*}\\
&\qquad+(I-\Theta W^{*})^{-1}\Theta D_{W}^{2}
\Theta^{*}(I-W\Theta^{*})^{-1}]D_{W^{*}},
\end{align*}

We have
\begin{align*}
&W\Theta^{*}(I-W\Theta^{*})^{-1}-(I-\Theta W^{*})^{-1}\Theta W^{*}
+(I-\Theta W^{*})^{-1}\Theta D_{W}^{2}
\Theta^{*}(I-W\Theta^{*})^{-1}\\
&\qquad=W\Theta^{*}(I-W\Theta^{*})^{-1}-(I-\Theta W^{*})^{-1}\Theta W^{*}\\
&\qquad\qquad+(I-\Theta W^{*})^{-1}\Theta (I-W^{*}W)
\Theta^{*}(I-W\Theta^{*})^{-1}\\
&\qquad=W\Theta^{*}(I-W\Theta^{*})^{-1}-(I-\Theta W^{*})^{-1}\Theta W^{*}\\
&\qquad\qquad+(I-\Theta W^{*})^{-1}(I-W\Theta^{*})^{-1}+(I-\Theta W^{*})^{-1}\Theta W^{*}W
\Theta^{*}(I-W\Theta^{*})^{-1}\\
&\qquad=W\Theta^{*}(I-W\Theta^{*})^{-1}-(I-\Theta W^{*})^{-1}(I-W\Theta^{*})^{-1}\\
&\qquad\qquad+(I-\Theta W^{*})^{-1}\Theta W^{*}
+(I-\Theta W^{*})^{-1}\Theta W^{*}W
\Theta^{*}(I-W\Theta^{*})^{-1}\\
&\qquad=[W\Theta^{*}-(I-\Theta W^{*})^{-1}](I-W\Theta^{*})^{-1}\\
&\qquad\qquad+(I-\Theta W^{*})^{-1}\Theta W^{*}(I-W\Theta^{*})(I-W\Theta^{*})^{-1}\\
&\qquad\qquad+(I-\Theta W^{*})^{-1}\Theta W^{*}W
\Theta^{*}(I-W\Theta^{*})^{-1}\\
&\qquad=[W\Theta^{*}-(I-\Theta W^{*})^{-1}](I-W\Theta^{*})^{-1}\\
&\qquad\qquad+(I-\Theta W^{*})^{-1}[\Theta W^{*}(I-W\Theta^{*})
+\Theta W^{*}W\Theta^{*}](I-W\Theta^{*})^{-1}\\
&\qquad=[W\Theta^{*}-(I-\Theta W^{*})^{-1}](I-W\Theta^{*})^{-1}\\
&\qquad\qquad+(I-\Theta W^{*})^{-1}\Theta W^{*}
(I-W\Theta^{*})^{-1}\\
&\qquad=[W\Theta^{*}-(I-\Theta W^{*})^{-1}+(I-\Theta W^{*})^{-1}\Theta W^{*}
](I-W\Theta^{*})^{-1}\\
&\qquad=[W\Theta^{*}-(I-\Theta W^{*})^{-1}(I-\Theta W^{*})
](I-W\Theta^{*})^{-1}\\
&\qquad=(W\Theta^{*}-I
)(I-W\Theta^{*})^{-1}=-(I-W\Theta^{*})
(I-W\Theta^{*})^{-1}=-I.
\end{align*}

Therefore 
\[
\Theta^{\prime}(e^{it})\Theta^{\prime *}(e^{it})=WW^*+D_{W^*}^2= I,
\]
and so $ \Theta' $ is inner. We leave to the reader to check that $ \Theta' $ is pure.
\end{proof} 
\begin{rmk}\label{eq: formula for theta}
The function $\Theta$ can be obtained from $\Theta^{\prime}$ as
\[
\Theta(\lambda)=W+ D_{W^{*}}(I+\Theta'(\lambda) W^{*})^{-1}\Theta'(\lambda) D_{W}.
\]
\end{rmk}

Let $K_{\Theta}$ be the model space corresponding to inner function $\Theta$ and $K_{\Theta^{'}}$ be model space corresponding to $\Theta^{'}$. We introduce now the generalized Crofoot transformation between these spaces.

\begin{thm}(Generalized Crofoot transformation) \label{GCT} Let $W$ be a strict contraction, $\Theta$ a pure inner function, and suppose $ \Theta' $ is defined by~\eqref{eq:formula for Theta'}. Then
 the map
$J_{W}$   defined by
$$
J_{W}f=D_{W^{*}}(I-\Theta(\lambda)W^{*})^{-1}f
$$
is a unitary operator from $ K_{\Theta}$ to $K_{\Theta^{'}}$.
\end{thm}

To prove Theorem \ref{GCT} we first prove the following proposition:
\begin{prop}\label{Eq: for Reproducing}
Let $y\in E$ and $\lambda\in \mathbb{D}$, then
\begin{equation}\label{eq:J(repro)}
J_{W}(k_{\lambda}^{\Theta}(I-W\Theta(\lambda)^{*})^{-1}D_{W^{*}}y)=k_{\lambda}^{\Theta^{\prime}}y,\quad
J_{W}(\widetilde{k_{\lambda}^{\Theta}}(I-W^{*}\Theta(\lambda))^{-1}D_{W}y)=\widetilde{k_{\lambda}^{\Theta^{\prime}}}y.
\end{equation}
\end{prop}
\begin{proof}
We have
\begin{align*}
&(I-\Theta^{\prime}(z)\Theta^{\prime}(\lambda)^{*})y\\
&\qquad=y-(-W+D_{W^{*}}(I-\Theta(z)W^{*})^{-1}\Theta(z)D_{W})\\
&\qquad\qquad(-W^{*}+D_{W}\Theta(\lambda)^{*}(I-W\Theta(\lambda)^{*})^{-1}D_{W^{*}})y\\
&\qquad=(I-WW^{*})y+WD_{W}\Theta(\lambda)^{*}(I-W\Theta(\lambda)^{*})^{-1}D_{W^{*}}y\\
&\qquad\qquad+D_{W^{*}}(I-\Theta(z)W^{*})^{-1}\Theta(z)D_{W}W^{*}y\\
&\qquad\qquad-D_{W^{*}}(I-\Theta(z)W^{*})^{-1}\Theta(z)D_{W}^{2}\Theta(\lambda)^{*}(I-W\Theta(\lambda)^{*})^{-1}D_{W^{*}}y\\
&\qquad=D_{W^{*}}^{2}y+D_{W^{*}}W\Theta(\lambda)^{*}(I-W\Theta(\lambda)^{*})^{-1}D_{W^{*}}y\\
&\qquad\qquad+D_{W^{*}}(I-\Theta(z)W^{*})^{-1}\Theta(z)W^{*}D_{W^{*}}y\\
&\qquad\qquad-D_{W^{*}}(I-\Theta(z)W^{*})^{-1}\Theta(z)D_{W}^{2}\Theta(\lambda)^{*}(I-W\Theta(\lambda)^{*})^{-1}D_{W^{*}}y\\
&\qquad=D_{W^{*}}[I+W\Theta(\lambda)^{*}(I-W\Theta(\lambda)^{*})^{-1}+(I-\Theta(z)W^{*})^{-1}\Theta(z)W^{*}\\
&\qquad\qquad-(I-\Theta(z)W^{*})^{-1}\Theta(z)D_{W}^{2}\Theta(\lambda)^{*}(I-W\Theta(\lambda)^{*})^{-1}]D_{W^{*}}y\\
&\qquad=D_{W^{*}}(I-\Theta(z)W^{*})^{-1}[(I-\Theta(z)W^{*})\\
&\qquad\qquad+(I-\Theta(z)W^{*})W\Theta(\lambda)^{*}(I-W\Theta(\lambda)^{*})^{-1}\\
&\qquad\qquad+\Theta(z)W^{*}-\Theta(z)D_{W}^{2}\Theta(\lambda)^{*}(I-W\Theta(\lambda)^{*})^{-1}]D_{W^{*}}y\\
&\qquad=D_{W^{*}}(I-\Theta(z)W^{*})^{-1}[I+(I-\Theta(z)W^{*})W\Theta(\lambda)^{*}(I-W\Theta(\lambda)^{*})^{-1}\\
&\qquad\qquad-\Theta(z)D_{W}^{2}\Theta(\lambda)^{*}(I-W\Theta(\lambda)^{*})^{-1}]D_{W^{*}}y\\
&\qquad=D_{W^{*}}(I-\Theta(z)W^{*})^{-1}[I+(I-\Theta(z)W^{*})W\Theta(\lambda)^{*}(I-W\Theta(\lambda)^{*})^{-1}\\
&\qquad\qquad-\Theta(z)(I-WW^{*})\Theta(\lambda)^{*}(I-W\Theta(\lambda)^{*})^{-1}]D_{W^{*}}y\\
\end{align*}

\begin{align*}
&=D_{W^{*}}(I-\Theta(z)W^{*})^{-1}[I+W\Theta(\lambda)^{*}(I-W\Theta(\lambda)^{*})^{-1}\\
&\qquad-\Theta(z)W^{*}W\Theta(\lambda)^{*}(I-W\Theta(\lambda)^{*})^{-1}-\Theta(z)\Theta(\lambda)^{*}(I-W\Theta(\lambda)^{*})^{-1}\\
&\qquad+\Theta(z)W^{*}W\Theta(\lambda)^{*}(I-W\Theta(\lambda)^{*})^{-1}]D_{W^{*}}y\\
&=D_{W^{*}}(I-\Theta(z)W^{*})^{-1}[I-W\Theta(\lambda)^{*}+W\Theta(\lambda)^{*}\\
&\qquad-\Theta(z)\Theta(\lambda)^{*}](I-W\Theta(\lambda)^{*})^{-1}D_{W^{*}}y\\
&=D_{W^{*}}(I-\Theta(z)W^{*})^{-1}[I-\Theta(z)\Theta(\lambda)^{*}](I-W\Theta(\lambda)^{*})^{-1}D_{W^{*}}y\\
&=J_{W}(I-\Theta(z)\Theta(\lambda)^{*})(I-W\Theta(\lambda)^{*})^{-1}D_{W^{*}}y.
\end{align*}
It follows that
$J_{W}(k_{\lambda}^{\Theta}(I-W\Theta(\lambda)^{*})^{-1}D_{W^{*}}y)
=k_{\lambda}^{\Theta^{\prime}}y.$

For the other equality, we have
     \begin{align*}
\widetilde{k_{\lambda}^{\Theta^{\prime}}}y
&=\frac{1}{z-\lambda}(\Theta^{\prime}(z)-\Theta^{\prime}(\lambda))y\\
&=\frac{1}{z-\lambda}[-W+D_{W^{*}}(I-\Theta(z) W^{*})^{-1}\Theta(z) D_{W}+W-D_{W^{*}}(I-\Theta(\lambda) W^{*})^{-1}\Theta(\lambda) D_{W}]y\\
&=\frac{1}{z-\lambda}[D_{W^{*}}(I-\Theta(z) W^{*})^{-1}\Theta(z) D_{W}-D_{W^{*}}(I-\Theta(\lambda) W^{*})^{-1}\Theta(\lambda) D_{W}]y\\
&=\frac{1}{z-\lambda}D_{W^{*}}[(I-\Theta(z) W^{*})^{-1}\Theta(z) -(I-\Theta(\lambda) W^{*})^{-1}\Theta(\lambda)] D_{W}y\\
&=\frac{1}{z-\lambda}D_{W^{*}}[(I-\Theta(z) W^{*})^{-1}\Theta(z)-(I-\Theta(z) W^{*})^{-1}\Theta(\lambda)
+(I-\Theta(z) W^{*})^{-1}\Theta(\lambda)\\
&-(I-\Theta(\lambda) W^{*})^{-1}\Theta(\lambda)] D_{W}y\\
&=\frac{1}{z-\lambda}D_{W^{*}}(I-\Theta(z) W^{*})^{-1}[\Theta(z)-(I-\Theta(z)W^{*})(I-\Theta(\lambda)W^{*})^{-1}\Theta(\lambda)] D_{W}y\\
&=\frac{1}{z-\lambda}D_{W^{*}}(I-\Theta(z) W^{*})^{-1}[\Theta(z)-(I-\Theta(\lambda)W^{*})^{-1}\Theta(\lambda)+\Theta(z)W^{*}(I-\Theta(\lambda)W^{*})^{-1}\Theta(\lambda)] D_{W}y\\
&=\frac{1}{z-\lambda}D_{W^{*}}(I-\Theta(z) W^{*})^{-1}[\Theta(z)-\Theta(\lambda)(I-W^{*}\Theta(\lambda))^{-1}+\Theta(z)W^{*}\Theta(\lambda)(I-W^{*}\Theta(\lambda))^{-1}] D_{W}y\\
&=\frac{1}{z-\lambda}D_{W^{*}}(I-\Theta(z) W^{*})^{-1}[\Theta(z)(I-W^{*}\Theta(\lambda))-\Theta(\lambda)+\Theta(z)W^{*}\Theta(\lambda)](I-W^{*}\Theta(\lambda))^{-1} D_{W}y\\
&=\frac{1}{z-\lambda}D_{W^{*}}(I-\Theta(z) W^{*})^{-1}[\Theta(z)-\Theta(\lambda)](I-W^{*}\Theta(\lambda))^{-1} D_{W}y\\
&=D_{W^{*}}(I-\Theta(z) W^{*})^{-1}(\frac{1}{z-\lambda}(\Theta(z)-\Theta(\lambda)))(I-W^{*}\Theta(\lambda))^{-1} D_{W}y\\
&=J_{W}(\widetilde{k_{\lambda}^{\Theta}}(I-W^{*}\Theta(\lambda))^{-1} D_{W}y.\qedhere
\end{align*}
\end{proof}

\begin{proof}[Proof of Theorem~\ref{GCT}]
First we claim that $J_{W}K_{\Theta}\subset K_{\Theta^{'}}$. To show that $J_{W}f$ belong
to $K_{\Theta^{'}}$ for every $f\in K_{\Theta}$, we must show that $J_{W}f$ is orthogonal to every function of the form $\Theta^{'}g$ where $g\in H^{2}(E)$.
This follows from the following computation.
Note here we use the fact that $ \Theta(e^{it})\Theta^*(e^{it})=\Theta^*(e^{it})\Theta(e^{it})=I $ almost everywhere on $\mathbb{T}$.
\begin{align*}
\langle J_{W}f, \Theta ^{'}g\rangle
&=\langle D_{W^{*}}(I-\Theta(e^{it})W^{*})^{-1}f, \Theta^{'}g\rangle\\
&=\langle f, (I-W\Theta(e^{it})^{*})^{-1}D_{W^{*}}\Theta^{'}g\rangle\\
&=\langle f, (I-W\Theta(e^{it})^{*})^{-1}D_{W^{*}}[-W+D_{W^{*}}(I-\Theta(e^{it}) W^{*})^{-1}\Theta(e^{it}) D_{W}]g\rangle \\
&=\langle f, [-(I-W\Theta^{*})^{-1}D_{W^{*}}W+(I-W\Theta^{*})^{-1}D_{W^{*}}^{2}(I-\Theta W^{*})^{-1}\Theta D_{W}]g\rangle\\
&=\langle f, [-(I-W\Theta^{*})^{-1}WD_{W}+(I-W\Theta^{*})^{-1}D_{W^{*}}^{2}(I-\Theta W^{*})^{-1}\Theta D_{W}]g\rangle\\
&=\langle f, (I-W\Theta^{*})^{-1}[-W+D_{W^{*}}^{2}\Theta(I-W^{*}\Theta )^{-1}] D_{W}g\rangle\\
&=\langle f, (I-W\Theta^{*})^{-1}[-W(I-W^{*}\Theta )(I-W^{*}\Theta )^{-1}+D_{W^{*}}^{2}\Theta(I-W^{*}\Theta )^{-1}] D_{W}g\rangle\\
&=\langle f, (I-W\Theta^{*})^{-1}[-W(I-W^{*}\Theta )+(I-WW^{*})\Theta](I-W^{*}\Theta )^{-1} D_{W}g\rangle\\
&=\langle f, (I-W\Theta^{*})^{-1}[-W+WW^{*}\Theta +\Theta -WW^{*}\Theta](I-W^{*}\Theta )^{-1} D_{W}g\rangle\\
&=\langle f, (I-W\Theta(e^{it})^{*})^{-1}[\Theta(e^{it}) -W](I-W^{*}\Theta(e^{it}) )^{-1} D_{W}g\rangle\\
&=\langle f, (I-W\Theta(e^{it})^{*})^{-1}(I -W\Theta(e^{it})^{*})\Theta(e^{it})(I-W^{*}\Theta(e^{it}) )^{-1} D_{W}g\rangle\\
&=\langle f, \Theta(e^{it})(I-W^{*}\Theta(e^{it}) )^{-1} D_{W}g\rangle\\
&=0,
\end{align*}
because the function $\Theta(e^{it})(I-W^{*}\Theta(e^{it}))^{-1} D_{W}g\in \Theta H^{2}(E)$. Hence it follows that $J_{W}K_{\Theta}\subset K_{\Theta^{'}}$.\par
Now define the operator $J_{W}^{'}: K_{\Theta^{\prime}}\longrightarrow K_{\Theta}$ by
\begin{equation}
J_{W}^{'}g=D_{W^{*}}(I+\Theta^{\prime}W^{*})^{-1}g, \quad \forall g\in K_{
\Theta^{\prime}}.
\end{equation}
First we show that $J_{W}^{'}K_{\Theta^{'}}\subset K_{\Theta}$. For this purpose we will prove that $J_{W}^{'}g$ is orthogonal to $\Theta h$ for any $g\in K_{\Theta^{'}}$ and any $h\in H^{2}(E)$. We have
\begin{align*}
\langle J_{W}^{'}g, \Theta h\rangle
&=\langle D_{W^{*}}(I+\Theta^{\prime}W^{*})^{-1}f, \Theta g\rangle=\langle f, (I+W\Theta^{\prime*})^{-1}D_{W^{*}}\Theta g\rangle\\
&=\langle f, (I+W\Theta^{\prime*})^{-1}D_{W^{*}}[W+D_{W^{*}}(I+\Theta^{\prime}) W^{*})^{-1}\Theta^{\prime} D_{W}]g\rangle \\
&=\langle f, (I+W\Theta^{\prime*})^{-1}[D_{W^{*}}W+D_{W^{*}}^{2}(I+\Theta^{\prime} W^{*})^{-1}\Theta^{\prime} D_{W}]g\rangle\\
&=\langle f, (I+W\Theta^{\prime*})^{-1}[WD_{W}+D_{W^{*}}^{2}(I+\Theta^{\prime} W^{*})^{-1}\Theta^{\prime} D_{W}]g\rangle\\
&=\langle f, (I+W\Theta^{\prime*})^{-1}[W+D_{W^{*}}^{2}\Theta^{\prime}(I+W^{*}\Theta^{\prime} )^{-1}] D_{W}g\rangle\\
&=\langle f, (I+W\Theta^{\prime*})^{-1}[W(I+W^{*}\Theta^{\prime} )+D_{W^{*}}^{2}\Theta^{\prime}](I+W^{*}\Theta^{\prime} )^{-1} D_{W}g\rangle\\
&=\langle f, (I+W\Theta^{\prime*})^{-1}[W(I+W^{*}\Theta^{\prime} )+(I-WW^{*})\Theta^{\prime}](I+W^{*}\Theta^{\prime} )^{-1} D_{W}g\rangle\\
&=\langle f, (I+W\Theta^{\prime*})^{-1}[W+WW^{*}\Theta^{\prime} +\Theta^{\prime} -WW^{*}\Theta^{\prime}](I+W^{*}\Theta^{\prime} )^{-1} D_{W}g\rangle\\
&=\langle f, (I+W\Theta^{\prime*})^{-1}[ \Theta^{\prime}+W](I+W^{*}\Theta^{\prime}) )^{-1} D_{W}g\rangle\\
&=\langle f, (I+W\Theta^{\prime*})^{-1}(I +W\Theta^{\prime*})\Theta^{\prime}(I+W^{*}\Theta^{\prime} )^{-1} D_{W}g\rangle\\
&=\langle f, \Theta^{\prime}(I+W^{*}\Theta^{\prime}) )^{-1} D_{W}g\rangle\\
&=0,
\end{align*}
and so $J_{W}^{'}K_{\Theta^{'}}\subset K_{\Theta}$. \par
Next we prove that $J_{W}^{'}$ is the inverse of $J_{W}$. If $f\in K_{\Theta}$, then
\begin{align*}
J_{W}^{'}J_{W}f
&=D_{W^{*}}(I+\Theta^{\prime}W^{*})^{-1}D_{W^{*}}(I-\Theta W^{*})^{-1}f\\
&=D_{W^{*}}[I+(-W+ D_{W*}(I-\Theta W^{*})^{-1}\Theta D_{W})W^{*}]^{-1}D_{W^{*}}(I-\Theta W^{*})^{-1}f\\
&=D_{W^{*}}[I-WW^{*}+ D_{W^{*}}(I-\Theta W^{*})^{-1}\Theta W^{*}D_{W^{*}}]^{-1}D_{W^{*}}(I-\Theta W^{*})^{-1}f\\
&=D_{W^{*}}[D_{W^{*}}^{2}+ D_{W^{*}}(I-\Theta W^{*})^{-1}\Theta W^{*}D_{W^{*}}]^{-1}D_{W^{*}}(I-\Theta W^{*})^{-1}f\\
&=D_{W^{*}}[D_{W^{*}}^{-2}+ D_{W^{*}}^{-1}W^{*-1}\Theta^{-1}(I-\Theta W^{*}) D_{W^{*}}^{-1}]D_{W^{*}}(I-\Theta W^{*})^{-1}f\\
&=[I+ W^{*-1}\Theta^{-1}(I-\Theta W^{*})](I-\Theta W^{*})^{-1}f\\
&=[I+ (\Theta W^{*})^{-1}(I-\Theta W^{*})](I-\Theta W^{*})^{-1}f\\
&=(I-\Theta W^{*})^{-1}f+ (\Theta W^{*})^{-1}f\\
&=(I-\Theta W^{*})^{-1}f+ (\Theta W^{*})^{-1}f+f-f\\
&=(I-\Theta W^{*})^{-1}f- (I-\Theta W^{*})^{-1}f+f=f.
\end{align*}
For $g\in K_{\Theta^{\prime}}$ we have
\begin{align*}
J_{W}J_{W}^{'}g
&=D_{W^{*}}(I-\Theta W^{*})^{-1}D_{W^{*}}(I+\Theta^{\prime}W^{*})^{-1}g\\
&=D_{W^{*}}(I-\Theta W^{*})^{-1}D_{W^{*}}[I+(-W+ D_{W^{*}}(I-\Theta W^{*})^{-1}\Theta D_{W})W^{*}]^{-1}g\\
&=D_{W^{*}}(I-\Theta W^{*})^{-1}D_{W^{*}}[I-WW^{*}+ D_{W^{*}}(I-\Theta W^{*})^{-1}\Theta D_{W}W^{*}]^{-1}g\\
&=D_{W^{*}}(I-\Theta W^{*})^{-1}D_{W^{*}}[D^{2}_{W^{*}}+ D_{W^{*}}(I-\Theta W^{*})^{-1}\Theta W^{*}D_{W^{*}}]^{-1}g\\
&=D_{W^{*}}(I-\Theta W^{*})^{-1}D_{W^{*}}[D^{-2}_{W^{*}}+ D_{W^{*}}^{-1}W^{*-1}\Theta^{-1} (I-\Theta W^{*})D_{W^{*}}^{-1}]g\\
&=D_{W^{*}}(I-\Theta W^{*})^{-1}[D^{-1}_{W^{*}}+ W^{*-1}\Theta^{-1} (I-\Theta W^{*})D_{W^{*}}^{-1}]g\\
&=D_{W^{*}}(I-\Theta W^{*})^{-1}[I+ W^{*-1}\Theta^{-1} (I-\Theta W^{*})]D_{W^{*}}^{-1}g\\
&=D_{W^{*}}(I-\Theta W^{*})^{-1} (\Theta W^{*})^{-1} D_{W^{*}}^{-1}g\\
&=D_{W^{*}}[(\Theta W^{*})^{-1}+(I-\Theta W^{*})^{-1}] D_{W^{*}}^{-1}g\\
&=D_{W^{*}}[I-I+(\Theta W^{*})^{-1}+(I-\Theta W^{*})^{-1}] D_{W^{*}}^{-1}g\\
&=D_{W^{*}}[I-(I-\Theta W^{*})^{-1}+(I-\Theta W^{*})^{-1}] D_{W^{*}}^{-1}g\\
&=g.
\end{align*}
The above computation shows that $J_{W}^{'}$ is the inverse of $J_{W}$ and $J_{W}K_{\Theta}=K_{\Theta^{'}}$.\par
We now show that $J_{W}$ is a unitary operator. By using Proposition \ref{Eq: for Reproducing} we obtain
\begin{align*}
\langle J_W k^{\Theta}_\lambda x, J_W k^{\Theta}_\mu y\rangle
&=\langle J_W k^{\Theta}_\lambda x,  k^{\Theta^{\prime}}_\mu D_{W^{*}}^{-1}(I-W\Theta^{*}(\mu))y\rangle\\
&=\langle J_{W} k^{\Theta}_\lambda(\mu)x,  D_{W^{*}}^{-1}(I-W\Theta^{*}(\mu))y\rangle\\
&=\langle D_{W^{*}}(I-\Theta(\mu))W^{*})^{-1}k^{\Theta}_\lambda(\mu)x,  D_{W^{*}}^{-1}(I-W\Theta^{*}(\mu))y\rangle\\
&=\langle(I-\Theta (\mu)W^{*})D_{W^{*}}^{-1} D_{W^{*}}(I-\Theta(\mu))W^{*})^{-1}k^{\Theta}_\lambda(\mu)x,  y\rangle\\
&=\langle k^{\Theta}_\lambda(\mu)x,  y\rangle=\langle  k^\Theta_\lambda x,   k^\Theta_\mu y\rangle.
\end{align*}
Therefore 
\[
\langle J_W f, J_W g\rangle= \langle f, g\rangle
\]
for any $ f,g $ in the linear span of $ k^{\Theta}_\lambda x $, $ \lambda\in\mathbb{D} $, $ x\in \mathbb{C^{d}} $.
The required result follows by the density of this last set in $ 
 K_\Theta $.
\end{proof}

\begin{rmk}
The defect spaces of $S_{\Theta^{\prime}}$ in terminology of \cite{NF} are given by
\begin{equation}
\begin{split}
\mathcal{D}^{\prime}_{*}&=\Big\{\frac{1}{z}(\Theta^{\prime}(z)-\Theta^{\prime}(0))x:x\in \mathbb{C}^d \Big\}\\
\mathcal{D}^{\prime}&=\{(I-\Theta^{\prime}(z)\Theta^{\prime}(0)^{*})x:x\in \mathbb{C}^{d}\}.
\end{split}
\end{equation}
\end{rmk}
\begin{corollary}\label{inclusion}
\item[(i)]$ f\in \mathcal{D_{*}}^{\perp}$ if and only if $J_{W}f\in \mathcal{D^{ \prime\perp}_{*}}$.
\item[(ii)] $g\in\mathcal{D^{ \prime\perp}_{*}}$ if and only if $J_{W}^{*}g\in \mathcal{D_{*}}^{\perp}$
\end{corollary}
\begin{proof}(i)
By using Proposition \ref{Eq: for Reproducing} we have
\begin{align*}
\langle J_{W}f, \widetilde{k_{0}^{\Theta^{\prime}}}x\rangle
&=\langle f, J_{W}^{*}\widetilde{k_{0}^{\Theta^{\prime}}}x\rangle
=\langle f, \widetilde{k_{0}^{\Theta}} D_{W}y\rangle=0.
\end{align*}
(ii) Let $f\in\mathcal{D^{ \prime\perp}_{*}}$ and $D_{W}y=x$ then by Proposition \ref{Eq: for Reproducing} we obtain
\begin{align*}
\langle J_{W}^{*}g, \widetilde{k_{0}^{\Theta}}x\rangle
&=\langle g, J_{W}\widetilde{k_{0}^{\Theta}}x\rangle=\langle g, \widetilde{k_{0}^{\Theta^{\prime}}}y\rangle=0.
\end{align*}
\end{proof}

\begin{prop}\label{main them}
Let $f\in K_{\Theta}$, we have
$$S_{\Theta^{'}}^{*}J_{W}f=J_{W}S_{\Theta}^{*}f+S_{\Theta^{'}}^{*}J_{W}f(0).$$
\end{prop}
\begin{proof}
Let $f\in K_{\Theta}$, then
\begin{align*}
S_{\Theta^{'}}^{*}J_{W}f
&=S_{\Theta^{'}}^{*}[D_{W^{*}}(I-\Theta(z)W^{*})^{-1}f]\\
&=\frac{1}{z}\Big(D_{W^{*}}(I-\Theta(z)W^{*})^{-1}f(z)-D_{W^{*}}(I-\Theta(0)W^{*})^{-1}f(0)\Big)\\
&=D_{W^{*}}\frac{1}{z}\Big((I-\Theta(z)W^{*})^{-1}f(z)-(I-\Theta(0)W^{*})^{-1}f(0)\Big)\\
&=D_{W^{*}}\frac{1}{z}\Big((I-\Theta(z)W^{*})^{-1}f(z)-(I-\Theta(z)W^{*})^{-1}f(0)\\
&+(I-\Theta(z)W^{*})^{-1}f(0)-(I-\Theta(0)W^{*})^{-1}f(0)\Big)\\
&=D_{W^{*}}(I-\Theta(z)W^{*})^{-1}\frac{1}{z}(f(z)-f(0))\\
&+D_{W^{*}}\frac{1}{z}((I-\Theta(z)W^{*})^{-1}f(0)-(I-\Theta(0)W^{*})^{-1}f(0))\\
&=D_{W^{*}}(I-\Theta(z)W^{*})^{-1}\frac{1}{z}(f(z)-f(0))\\
&+\frac{1}{z}(D_{W^{*}}(I-\Theta(z)W^{*})^{-1}f(0)-D_{W^{*}}(I-\Theta(0)W^{*})^{-1}f(0))\\
&=D_{W^{*}}(I-\Theta(z)W^{*})^{-1}S_{\Theta}^{*}f+S_{\Theta^{'}}^{*}(D_{W^{*}}(I-\Theta(z)W^{*})^{-1}f(0))\\
&=J_{W}S_{\Theta}^{*}f+S_{\Theta^{'}}^{*}J_{W}f(0).
\end{align*}
\end{proof}

\begin{lem}\label{commut}
$S_{\Theta^{\prime}}J_{W}f=J_{W}S_{\Theta}f$ for $f\in \mathcal{D}^{\perp}_{*}$.
\end{lem}
\begin{proof}
Let $ f\in \mathcal{D}^{\perp} $; so $ f\perp k^\Theta_0 x $ for any $ x\in\mathbb{C^{d}} $, which by the reproducing kernel property of $ k^\Theta_0 $ is equivalent to $f(0)=0 $. So from Proposition~\ref{main them} it follows that
\begin{equation}\label{eq:commutation on D perp}
S_{\Theta^{'}}^{*}J_{W}f=J_{W}S_{\Theta}^{*}f \text{ for } f\in \mathcal{D}^\perp.
\end{equation}
Now by (\ref{eq: S Theta}), it follows that $ S_\Theta^* $ is a unitary (division by $ z $) from $ \mathcal{D}^\perp $ to $ \mathcal{D}_*^\perp $ (and similarly for $ \Theta' $). On the other hand, from Proposition~\ref{Eq: for Reproducing} it follows that $ J_W $ maps (unitarily) $ \mathcal{D}^\perp $ to $ \mathcal{D}'^\perp $, and $ \mathcal{D}_*^\perp $ to $ \mathcal{D}_*'{}^\perp $. Using~\eqref{eq:commutation on D perp}, we have the following commutative diagram of unitary operators:
\[
\begin{tikzcd}
\mathcal{D}^\perp \arrow[r, "S_\Theta^*"]\arrow[d, "J_W"]& \mathcal{D}_*^\perp\arrow[d, "J_W"] \\
\mathcal{D}'{}^\perp\arrow[r, "S_{\Theta'}^*" ]& \mathcal{D}'_*{}^\perp.
\end{tikzcd}
\]
From the operators in above diagram as acting between these spaces, we have \[S_{\Theta^{'}}^{*}J_{W}=J_{W}S_{\Theta}^{*}  ;\]
by passing to the adjoint
we get
\[
J_{W}^*S_{\Theta^{'}}=S_{\Theta}J_{W}^*,
\]
where the two sides act from $\mathcal{D}'_*{}^\perp  $ to $\mathcal{D}^\perp$, and then multiplying on the left and on the right with $ J_W $,
\[
S_{\Theta^{'}}J_W=J_W S_{\Theta},
\]
where the two sides act from $ \mathcal{D}_*^\perp $ to $ \mathcal{D}'{}^\perp $. This completes the proof.
\end{proof}

A characterization of matrix valued truncated Toeplitz operators is obtained (see Theorem 5.5 in \cite{RT}) by shift invariance. A bounded operator
$A$ on $K_{\Theta}$ is called shift invariant if
$$f, Sf\in K_{\Theta} \quad implies \quad Q_{A}(Sf)=Q_{A}(f),$$
where $Q_{A}$ is associated quadratic form on $K_{\Theta}$ defined by $Q_{A}(f)=\langle Af, f\rangle$. It is well known that $S_{\Theta}f\in K_{\Theta}$
if and only if $f\in \mathcal{D}_{*}^{\perp}$.
\begin{thm}\cite{RT}\label{shift invariant thm}
A bounded operator $A$ on $K_{\Theta}$ is a matrix valued truncated Toeplitz operator  if and only if $A$ is shift invariant.
\end{thm}
 
The spaces of matrix valued truncated Toeplitz operators on $K_{\Theta}$ and $K_{\Theta^{\prime}}$ are denoted respectively by $\mathcal{T}_{\Theta}$
and $\mathcal{T}_{\Theta^{\prime}}$. The next result shows the action of the generalized Crofoot transform.
 
\begin{thm}\label{th:main theorem}
$\mathcal{T}_{\Theta}=J_{W}^{*}\mathcal{T}_{\Theta^{\prime}}J_{W}.$
\end{thm}
\begin{proof}
 Let $A\in \mathcal{T}_{\Theta^{\prime}}$, then $J_{W}^{*}AJ_{W}\in
J_{W}^{*}\mathcal{T}_{\Theta^{\prime}}J_{W}$. We shall show that $J_{W}^{*}AJ_{W}\in \mathcal{T}_{\Theta}$.
Assume that $f\in \mathcal{D}_{*}^{\perp}$ then
by Corollary \ref{inclusion} we have
$J_{W}f\in \mathcal{D^{\prime \perp}_{*}}$. By Lemma \ref{commut} we obtain
\begin{align*}
Q_{J_{W}^{*}AJ_{W}}(f)
&=\langle J_{W}^{*}AJ_{W}f, f\rangle=\langle AJ_{W}f, J_{W}f\rangle\\
&=\langle A S_{\Theta^{\prime}}J_{W}f, S_{\Theta^{\prime}}J_{W}f\rangle=\langle AJ_{W}S_{\Theta}f, J_{W}S_{\Theta}f\rangle\\
&=\langle J_{W}^{*}AJ_{W}S_{\Theta}f, S_{\Theta}f\rangle=Q_{J_{W}^{*}AJ_{W}}(S_{\Theta}f).
\end{align*}
It shows that $J_{W}^{*}AJ_{W}\in \mathcal{T}_{\Theta}$. Therefore by Theorem \ref{shift invariant thm} we obtain $J_{W}^{*}\mathcal{T}_{\Theta^{\prime}}J_{W}\subset \mathcal{T}_{\Theta}$.\par
To prove the required equality we now prove the inclusion $J_{W}\mathcal{T}_{\Theta}J_{W}^{-1}\subset\mathcal{T}_{\Theta^{\prime}}$.\par
Assume that $B\in \mathcal{T}_{\Theta}$ then we have $J_{W}B J_{W}^{*}\in J_{W}\mathcal{T}_{\Theta}J_{W}^{*}$. Let $f\in \mathcal{D}^{\prime \perp}_{*}$
then by Corollary \ref{inclusion} we get $J_{W}^{*}f\in \mathcal{D}^{\perp}_{*}$ and again by Lemma \ref{commut} we have
\begin{align*}
Q_{J_{W}B J_{W}^{*}}(f)
&=\langle J_{W}B J_{W}^{*}f, f\rangle=\langle B J_{W}^{*}f, J_{W}^{*}f\rangle\\
&=\langle BS_{\Theta} J_{W}^{*}f, S_{\Theta}J_{W}^{*}f\rangle=\langle B J_{W}^{*}S_{\Theta^{\prime}}f, J_{W}^{*}S_{\Theta^{\prime}}f\rangle\\
&=\langle J_{W} B J_{W}^{*}S_{\Theta^{\prime}}f, S_{\Theta^{\prime}}f\rangle=Q_{J_{W}B J_{W}^{*}}(S_{\Theta^{\prime}}f).
\end{align*}
Hence $J_{W}B J_{W}^{*}$ is shift invariant. Again by Theorem \ref{shift invariant thm} we have $J_{W}\mathcal{T}_{\Theta}J_{W}^{*}\subset\mathcal{T}_{\Theta^{\prime}}$ which implies that
$\mathcal{T}_{\Theta}\subset J_{W}^{*}\mathcal{T}_{\Theta^{\prime}}J_{W}$.   The required result follows.
\end{proof}
\section{Conjugation and Crofoot transform}
A bounded linear operator $T$ on a separable Hilbert space $E$ is complex symmetric if there exist an orthonormal basis for $E$ with respect to which $T$ has self-transpose matrix representation. An equivalent definition also exist and involve conjugation.
A \emph{conjugation} on a Hilbert space $E$ is a conjugate-linear, isometric and involutive map. We say that $T$ is $C$-symmetric if $T=CT^{*}C$, and complex symmetric if there exist a conjugation $C$ with respect to which $T$ is $C$-symmetric (see~\cite{GP}).\par
Let $\Gamma$ be a conjugation on $E$ and $\Theta$ is $\Gamma-$symmetric a.e on $\mathbb{T}$. Then the map $C_{\Gamma}:L^{2}(E)\longrightarrow L^{2}(E)$ defined by
$$C_{\Gamma}f=\Theta e^{-it}\Gamma f,$$ is conjugation on $L^{2}(E)$.
The following  lemma shows the relation, in this case, between conjugation and model spaces.
\begin{lem}\cite{RT}
Suppose that $\Gamma\Theta\Gamma=\Theta^{*}$ a.e on $\mathbb{T}$. Then $C_{\Gamma}K_{\Theta}=K_{\Theta}$.
\end{lem}

Note that in the scalar case the inner function $ \theta $ is always C-symmetric with respect to usual complex conjugation, which produces the standard conjugation on the model space $ K_\theta $.

Suppose that $\Gamma W^{*}=W\Gamma$  and $\Gamma\Theta\Gamma=\Theta^{*}$, then a simple calculation
shows that $\Gamma\Theta^{'}\Gamma=\Theta^{*'}$, and the relation $\Gamma D_{W^{*}}=D_{W}\Gamma$ also holds.

\begin{lem}
Suppose $C_{\Gamma}$ is conjugation on $K_{\Theta}$ and $C_{\Gamma}^{'}$ is conjugation on $K_{\Theta^{'}}$.
Then generalized Crofoot transformation intertwines the conjugation on $K_{\Theta}$ with the conjugation on $K_{\Theta^{'}}$,
that is $J_{W}C_{\Gamma}=C_{\Gamma}^{'}J_{W}$.
\end{lem}

\begin{proof}
Let $f\in K_{\Theta}$, then we have
\begin{align*}
C_{\Gamma}^{'}J_{W}f
&=\Theta^{'}e^{-it}\Gamma(D_{W^{*}}(I-\Theta W^{*})^{-1}f)\\
&=e^{-it}[-W+D_{W^{*}}(I-\Theta W^{*})^{-1}\Theta D_{W}]\Gamma(D_{W^{*}}(I-\Theta W^{*})^{-1}f)\\
&=e^{-it}[-W\Gamma D_{W^{*}}(I-\Theta W^{*})^{-1}f+D_{W^{*}}(I-\Theta W^{*})^{-1}\Theta D_{W}\Gamma D_{W^{*}}(I-\Theta W^{*})^{-1}f]\\
&=e^{-it}[-W D_{W}\Gamma(I-\Theta W^{*})^{-1}f+D_{W^{*}}(I-\Theta W^{*})^{-1}\Theta D_{W}^{2}\Gamma(I-\Theta W^{*})^{-1}f]\\
&=e^{-it}[-D_{W^{*}}W\Gamma(I-\Theta W^{*})^{-1}f+D_{W^{*}}(I-\Theta W^{*})^{-1}\Theta D_{W}^{2}\Gamma(I-\Theta W^{*})^{-1}f]\\
&=e^{-it}D_{W^{*}}[-W+(I-\Theta W^{*})^{-1}\Theta D_{W}^{2}]\Gamma(I-\Theta W^{*})^{-1}f\\
&=e^{-it}D_{W^{*}}[-(I-\Theta W^{*})^{-1}(I-\Theta W^{*})W+(I-\Theta W^{*})^{-1}\Theta D_{W}^{2}]\Gamma(I-\Theta W^{*})^{-1}f\\
&=e^{-it}D_{W^{*}}(I-\Theta W^{*})^{-1}[-(I-\Theta W^{*})W+\Theta D_{W}^{2}]\Gamma(I-\Theta W^{*})^{-1}f\\
&=e^{-it}D_{W^{*}}(I-\Theta W^{*})^{-1}[-W+\Theta W^{*}W+\Theta -\Theta W^{*}W]\Gamma(I-\Theta W^{*})^{-1}f\\
&=e^{-it}D_{W^{*}}(I-\Theta W^{*})^{-1}(\Theta -W)\Gamma(I-\Theta W^{*})^{-1}f\\
&=e^{-it}D_{W^{*}}(I-\Theta W^{*})^{-1}\Theta( I-\Theta^{*} W)\Gamma(I-\Theta W^{*})^{-1}f,
\end{align*}
since $( I-\Theta^{*} W)\Gamma(I-\Theta W^{*})^{-1}=\Gamma$ therefore we have
\begin{align*}
C_{\Gamma}^{'}J_{W}f
&=e^{-it}D_{W^{*}}(I-\Theta W^{*})^{-1}\Theta( I-\Theta^{*} W)\Gamma(I-\Theta W^{*})^{-1}f\\
&=e^{-it}D_{W^{*}}(I-\Theta W^{*})^{-1}\Theta \Gamma f \\
&=D_{W^{*}}(I-\Theta W^{*})^{-1}\Theta e^{it} \Gamma f\\
&=J_{W}C_{\Gamma}f.\qedhere
\end{align*}
\end{proof}


\end{document}